\documentclass[a4paper,10pt]{article}
\usepackage[english]{babel}
\usepackage{graphicx}
\usepackage[english]{babel}
\usepackage{amsmath}
\usepackage{amssymb}
\usepackage{amsfonts}
\usepackage{enumerate}
\usepackage{pstricks}
\usepackage{cite}
\usepackage{hyperref}
\usepackage{float}
\usepackage{comment}
\usepackage{authblk}

\usepackage{accents}

\usepackage[top=2cm, bottom=2cm, left=2cm, right=2cm]{geometry}

\newcommand*{\m}[1]{\mathbf{#1}}

\DeclareMathOperator{\RE}{Re}
\DeclareMathOperator{\IM}{Im}
\DeclareMathOperator{\SC}{Sc}
\DeclareMathOperator{\VEC}{Vec}

\newtheorem{theorem}{Theorem}[section]
\newtheorem{thm}[theorem]{Theorem}
\newtheorem{prop}[theorem]{Proposition}
\newtheorem{cor}[theorem]{Corollary}
\newtheorem{defn}[theorem]{Definition}

\newtheorem{lem}[theorem]{Lemma}
\newtheorem{rem}[theorem]{Remark}
\newtheorem{example}[theorem]{Example}
\newenvironment{proof}[1][Proof]{\textbf{#1: } }{\begin{flushright}$\blacksquare$\end{flushright}}

\begin{document}
	
	\title{Towards a Function Theory of Complexified Octonions}
	\author{Rolf Sören Krau\ss har\footnote{Chair of Mathematics, Faculty of Educational Sciences, University of Erfurt, Norhh\"auser Str. 63, D-99089 Erfurt, Germany}, Heikki Orelma\footnote{Tampere Institute of Mathematics, Hepolamminkatu 51, 33720 Tampere, Finland
			E-Mail: heikki.orelma@proton.me (Corresponding author)}}
	\date{\today}

	\maketitle
	
	\begin{abstract}
		In this article we study function theory in the $16$-dimensional space of complexified octonions $\mathbb{O}_{\mathbb{C}}=\mathbb{C}\otimes\mathbb{O}$. We define the complexified octonionic Cauchy–Riemann operator
		\[
		D=\sum_{j=0}^7 e_j\partial_{z_j}
		\]
		where $\partial_{z_j}:=\partial_{x_j}+i\partial_{y_j}$ for $j=0,1,...,7$. This operator together with its octonionic conjugate factorize the ultrahyperbolic operator
		\[
		\mathcal{L}=D\overline{D}=\overline{D}D=\sum_{j=0}^7  \partial_{z_j}^2.
		\]
		In real coordinates we obtain a Lichnerowicz-Weitzenb\"ock type formula of the form  
		\[
		\mathcal{L} = \Delta_x  - \Delta_y  + 2i \langle \partial_x, \partial_y \rangle
		\]
		One main goal of this paper consists in investigating the fundamental solution and polynomial solutions of the operator $\mathcal{L}$. The set of polynomial solutions is completely described by proposing explicit basis constructions and dimension formulae.  
		Another main goal that is to carefully figure out in detail how complexified octonion analysis precisely differs from classical real octonion analysis in $8$ dimensions. We  explain some really substantial differences. In particular, the component functions are not Euclidean harmonic but ultrahyperbolic harmonic exhibiting a very different regularity behavior. Furthermore, the classical Fischer-decomposition is not anymore true in this context. 
		
	\end{abstract}
	Mathematics Subject Classification (2020): Primary 30G35; Secondary 17A70\\
	\\
	Keywords: Complexified octonions, Octonion analysis, Cauchy–Riemann operator, Ultrahyperbolic operator, Fundamental solutions, Lichnerowics-Weitzenb\"ock formulae
	
	\section{Introduction}
	Hypercomplex number systems offer algebraic extensions of the complex number field to higher dimensions and have an enormous range of applications in pure and applied mathematics, in particular in mathematical analysis and mathematical physics. Many physical problems are related to elliptic, parabolic or hyperbolic partial differential equations where related Laplacians appear. 
	
	The famous theorem of Frobenius and Hurwitz tell us that the only four normed division algebras over $\mathbb{R}$ are the real numbers $\mathbb{R}$ itself, the two-dimensional complex number field $\mathbb{C}$, the four-dimensional Hamiltonian quaternionic skew field $\mathbb{H}$ and the eight-dimensional Cayley octonions $\mathbb{O}$, see for instance the famous overview report from J. Baez \cite{Baez}. 
	In fact, such as explained there in detail, these number systems exactly arise by applying the well-known Cayley doubling process to the reals where in each step of doubling these number fields of real dimension $2^n$, $n=1,2,\ldots$ arise. However, already going beyond the complex numbers fundamental algebraic properties successively get lost. In the context of the four-dimensional real quaternions the multiplication operation is not commutative anymore. Next, in the eight-dimensionnal  octonionic setting even the associativity is lost, although one still has an alternative composition algebra. These weaker properties also get lost in the next doubling process which leads to the 16-dimensional sedenions. These do not form a normed algebra, either, the only basic property that remains is the power associativity. 
	
	The octonionic setting thus represents a very special setting, because octonions are the largest normed Cayley algebras over the reals with the property that every non-zero element has a multiplicative inverse. 
	
	It should be mentioned that besides the Cayley-Dickson doubling there are other constructions of higher dimensional algebraic extensions of the complex numbers such as for example Clifford algebras. Clifford algebras are in general associative and they include $\mathbb{R}$, $\mathbb{C}$ and $\mathbb{H}$ as special subcases, but in all others one has to deal with zero-divisors also when considering other multiplication rules even in the low dimensional settings of bi-complex numbers or bi-quaternions. 
	
	Further options are to consider tensor products of one of these algebras with $\mathbb{C}$, $\mathbb{H}$ and so on. Forming the tensor product with $\mathbb{C}$ naturally leads to the complexified quaternions, the complexified octonions or complexified Clifford algebras. 
	
	If we want to use all these algebras in analysis, then one needs to develop tools of an associated function theory. Since the 1930s the development of quaternionic analysis and Clifford analysis has emerged to a main stream topic in higher dimensional function theory, cf. \cite{Bosshard1939,Fueter1949,BDS}. 
	Already the Fueter school considered for the first time in the 1940s functions in complexified Clifford algebras (see \cite{Fueter1949} and the references therein). An impactful  breakthrough however was achieved in the 1980s by J. Ryan with his paper \cite{Ryan} on complexfied Clifford analysis. In fact complexified Clifford analysis combined successfully tools from complex analysis with tools from Clifford analysis. More than twenty years later Hermitian Clifford analysis extended again this line of research in a pioneering way (see e.g. \cite{SaSo2002,BSS}). Recently one even started to investigate the combination of Clifford algebras with Cayley algebras. 
	
	On the other hand one also observed a strong effort in the investigation on the level of developing a function theory in the non-associative octonions. The first contribution in this direction was provided by P. Dentoni and M. Sce in 1973, cf. \cite{DS} and about fifteen years later also by K. N\^{o}no in 1988, see \cite{Nono}. A major breakthrough on octonionic analysis then was provided by the school of Xingmin Li L. Peng and their collaborators etc. up from 2000, presenting in particular a generalization of the Cauchy integral formula and Taylor and Laurent representation theorems, see for example \cite{XL2000,XL2002,XLT2008}. More recently J. Wang and X. Li \cite{WL2018,WL2020}, R.S. Krau{\ss}har et al. \cite{CKS2023,ConKra2021,KraMMAS,KFVR}, B. Prather \cite{Prather2021}, Q. Huo and Guangbin Ren \cite{QR2021,QR2022} also started to study function spaces in the octonionic setting attached with its analytical concepts. See also \cite{FL} where basic functional analytic concepts together with potential applications to some PDE have been outlined.     
	\par\medskip\par
	
	A natural next step is to develop basic function theoretical tools for the setting of complex octonions which is the overall goal of this paper. More precisely, we consider functions defined in open subsets of $\mathbb{R}^{16}$ which take values in $\mathbb{O}_{\mathbb{C}} := \mathbb{O} \otimes \mathbb{C}$. We put the focus on such functions that are annihilated at interior points of $\Omega$ by the complexified octonionic Cauchy-Riemann  $D := \frac{\partial }{\partial z_0} + \sum\limits_{j=1}^7 e_j \frac{\partial }{\partial z_j} $. The latter can be written in the form $D = \partial_x + i \partial_y$ where $i$ is the complex imaginary unit of $\mathbb{C}$ intertwining with each octonionic unit $e_j$ and $\partial_x$ and $\partial_y$ are the usual (real) octonionic Cauchy-Riemann operators considered earlier in the above mentioned papers, such as for instance in \cite{XL2000,XL2002}. The first feature we observe is that the octonionic conjugated operator $\overline{D}$ defined by $\overline{D} :=  \frac{\partial }{\partial z_0} - \sum\limits_{j=1}^7 e_j \frac{\partial }{\partial z_j}$ now gives rise to a Lichnerowicz-Weitzenb\"ock type formula  factorizing the ultrahyperbolic Laplace operator viz $D \overline{D}f = \Delta_x f - \Delta_y f +2i \langle \partial_x, \partial_y \rangle$. In turn the Hermitean conjugate (involving the octonionic conjugation and the complex conjugation at the same time) produces a Lichnerowicz-Weitzenb\"ock type formula involving the Euclidean Laplacian of $\mathbb{R}^{16}$, namely $D D^{+} = \Delta_x f + \Delta_y f +i(\partial_y(\partial_{\overline{x}} f)-\partial_y(\partial_{\overline{x}}f))$. This, together with the arising different pairs of Cauchy-Riemann operators, is carefully studied in Section~3 of this paper. 
	In Section~4 we next study the fundamental solution of the operator ${\cal{L}} := \Delta_x - \Delta_y +2i \langle \partial_x,\partial_y\rangle$. This is done using the Fourier transform considering the operator in polar coordinates. We carefully study the regularity properties of the fundamental solution which in particular allows us conclude that it really produces the Dirac $\delta$-function. 
	
	Finally in Section~5 we round off our paper by a study of polynomial solutions of this operator and provide dimension formulas together with explicit basis functions. In contrast to the associative Clifford case the standard Fischer-decomposition does not hold in our context - exposing clearly that complexified octonionic analysis is essentially different from (complexified) Clifford analysis. 
	
	Since the complexified octonions turned recently out to play a crucial role in theoretical physics, namely in developing extensions of the standard model of particle physics in a way also incorporating gravitational effects, see for example \cite{Burdik},  a basic toolkit for a complexified octonionic analysis will be of importance for further investigations in these directions. Note that recently also tensor products of Clifford algebras and Cayley Dickson-algebras are attracting current interest, see for instance \cite{Dinh}. As one future perspective the methods used in this paper here can serve as a roadmap to also develop the basic function theoretical tools in the context of general ``CD-Clifford algebras'' in which one has to combine all these toolkits together, such as also already theoretically proposed in \cite{LRW}. Investigations in this direction belong to the hot topics of current research interests and may now additionally give important impulses in the study of ultrahyperbolic PDE in analysis and mathematical physics.

	\section{Preliminaries}
	\subsection{The octonion algebra}
	The octonion algebra $\mathbb{O}$ is an $8$-dimensional non-commutative and non-associative unital algebra with unit element, generated by the generators
	\[
	\{ e_0,e_1,...,e_7\},
	\]
	where $e_0=1$ is the multiplicative identity. The algebra is a real algebra, so an arbitrary element $x\in\mathbb{O}$ can be written in the form
	\[
	x=x_0+x_1e_2+\cdots+x_7e_7,
	\]
	where $x_j\in\mathbb{R}$ for each $j=0,1,...,7$. The octonion algebra is thus an algebra where multiplication $\mathbb{O}\times\mathbb{O}\to\mathbb{O}$ is defined under the condition that it is unital and its norm is multiplicative. These requirements come at the cost that the algebra is no longer associative or commutative. There is considerable freedom in defining the multiplication, and in total there are exactly $480$ distinct ways to define multiplication on the octonions. In this research article, we use the following choice, which is one of the most common ones and has an important property for our upcoming calculations.    
	\begin{center}
		\begin{tabular}{c|cccccccc}
			& $1$   & $e_1$  & $e_2$  & $e_3$  & $e_4$  & $e_5$  & $e_6$  & $e_7$\\
			\hline
			$1$     & $1$   & $e_1$  & $e_2$  & $e_3$  & $e_4$  & $e_5$  & $e_6$  & $e_7$\\
			$e_1$   & $e_1$ & $-1$ & $e_3$  & $-e_2$ & $e_5$  & $-e_4$ & $-e_7$ & $e_6$\\
			$e_2$   & $e_2$ & $-e_3$ & $-1$ & $e_1$  & $e_6$  & $e_7$  & $-e_4$ & $-e_5$\\
			$e_3$   & $e_3$ & $e_2$  & $-e_1$ & $-1$ & $e_7$  & $-e_6$ & $e_5$  & $-e_4$\\
			$e_4$   & $e_4$ & $-e_5$ & $-e_6$ & $-e_7$ & $-1$   & $e_1$  & $e_2$  & $e_3$\\
			$e_5$   & $e_5$ & $e_4$  & $-e_7$ & $e_6$  & $-e_1$ & $-1$ & $-e_3$ & $e_2$\\
			$e_6$   & $e_6$ & $e_7$ & $e_4$  & $-e_5$ & $-e_2$ & $e_3$ & $-1$  & $-e_1$\\
			$e_7$   & $e_7$ & $-e_6$ & $e_5$  & $e_4$ & $-e_3$ & $-e_2$ & $e_1$  & $-1$
		\end{tabular}
	\end{center}
	From the multiplication table, one can see that the generators anticommute with each other, i.e., $e_ie_j=-e_je_i$ for $i,j=1,...,7$ and $i\neq j$. Moreover, $e_j^2=-1$ for each $j=1,...,7$. The $2\times 2$ subtable in the upper left corner of the multiplication table shows that $\{ 1,e_1\}$ generate the subalgebra $\mathbb{C}$ of complex numbers. Similarly, the $4\times 4$ subtable in the upper left corner shows that $\{1,e_1,e_2,e_3\}$ generate the quaternion subalgebra $\mathbb{H}$. The generator $e_4$ plays a special role, since when multiplied from the right, the set of generators ${e_4,e_5,e_6,e_7}$ becomes a set of quaternion generators, which results in a $4\times 4$  table of quaternions appearing in the lower right corner. Hence, an arbitrary octonion $x\in\mathbb{O}$ can be written in the form
	\begin{align}\label{persmuoto}
		x=a+be_4,    
	\end{align}
	where $a=x_0+x_1e_1+x_2e_3+e_3e_3$ and $b=x_4+x_4e_1+x_5e_2+x_6e_3+x_7e_3$ are quaternions. This enables the use of quaternions when computing with octonions, allowing one to avoid the issues caused by non-associativity, while not having to reduce all calculations to the component level, which would otherwise become overly chaotic. Computing with octonions is therefore based on the following lemma:
	\begin{lem}[\cite{KO}]\label{lem:paravecprod}
		Let $a,b\in\mathbb{H}$, then
		\begin{enumerate}
			\item[(a)] $e_4a=\overline{a}e_4$,
			\item[(b)] $e_4(ae_4)=-\overline{a}$,
			\item[(c)] $(ae_4)e_4=-a$,
			\item[(d)] $a(be_4)=(ba)e_4$,
			\item[(e)] $(ae_4)b=(a\overline{b})e_4$,
			\item[(f)] $(ae_4)(be_4)=-\overline{b}a$,
		\end{enumerate}
		where the overline denotes the quaternionic conjugate.
	\end{lem}
	By the previous lemma, the product of octonions is given by:
	\begin{lem}[\cite{KO}]\label{tulo}
		Let $x=a+be_4$ and $y=c+de_4$ be octonions in the quaternionic form. Then their product in quaternionic form is
		\begin{equation*}
			xy=(ac-\overline{d}b)+(b\overline{c}+da)e_4.
		\end{equation*}
	\end{lem}

	Using the quaternionic conjugate, one can define the octonionic conjugate for octonions of the form (\ref{persmuoto}) by setting
	\[
	\overline{x}:=\overline{a}-be_4.
	\]
	As with quaternions, an octonion $x\in\mathbb{O}$ can be written in the form
	\begin{align}\label{kvaternio}
		x=x_0+\m{x},
	\end{align}
	where $\m{x}=x_1e_1+\cdots+x_7e_7$. Then the conjugate is
	\[
	\overline{x}=x_0-\m{x}.
	\]
	We prove the following result as an example of how to compute with octonions.
	
	\begin{prop}
		Let $x,y\in\mathbb{O}$. Then:
		\begin{itemize}
			\item[(a)] $\overline{\overline{x}}=x$,
			\item[(b)] $\overline{xy}=\overline{y}\, \overline{x}$.
		\end{itemize}
	\end{prop}
	\begin{proof} Part (a) is obvious. We prove part (b). Let $x=a+be_4$ and $y=c+de_4$, where $a,b,c,d\in\mathbb{H}$. Then, using Lemma \ref{tulo} and the conjugation rules for quaternions, we obtain
		\begin{align*}
			\overline{xy}&=\overline{(ac-\overline{d}b)}-(b\overline{c}+da)e_4\\
			&= (\overline{c}\,\overline{a}-\overline{b}d)-(b\overline{c}+da)e_4\\
			&=(\overline{c}\,\overline{a}-\overline{b}d)+((-d)a+(-b)\overline{c})e_4.\\
			&=(\overline{c}-de_4)(\overline{a}-be_4)\\
			&=\overline{y}\, \overline{x}.
		\end{align*}
	\end{proof}
	For quaternions of the form (\ref{kvaternio}), the scalar and vector parts are defined by setting $\SC(x):=x_0$ and $\VEC(x):=\m{x}$. These can be computed using the \textit{projection formulas}
	\[
	\SC(x)=\frac{1}{2}(x+\overline{x})\ \text{ and }\ \VEC(x)=\frac{1}{2}(x-\overline{x}).
	\]
	Next we study the scalar part of the octonionic product.
	\begin{prop}
		Let $x=a+be_4$ and $y=c+de_4$ be octonions. Then
		\[
		\SC(xy)=\SC(ac)+\SC(\overline{b}d).
		\]
	\end{prop}
	\begin{proof}
		We compute using the projection formulas:
		\begin{align*}
			\SC(xy)&=\frac{1}{2}(xy+\overline{y}\, \overline{x})\\
			&= \frac{1}{2}\big( (ac-\overline{d}b)+(b\overline{c}+da)e_4+(\overline{c}\,\overline{a}-\overline{b}d)-(da+b\overline{c})e_4\big)\\
			&=\frac{1}{2} (ac+\overline{c}\,\overline{a}-\overline{d}b-\overline{b}d)\\
			&=\frac{1}{2}(ac+\overline{c}\,\overline{a})-\frac{1}{2}(\overline{b}d+\overline{d}b) \\
			&=\SC(ac)+\SC(\overline{b}d).
		\end{align*}
	\end{proof}
	This allows us to define the inner product by setting
	\[
	\langle x,y\rangle :=\SC(x\overline{y}) =\SC(a\overline{c})+\SC(\overline{b}d)=\sum_{j=0}^7 x_jy_j.
	\]
	Thus, the inner product can be expressed using the projection formulas as
	\[
	\langle x,y\rangle= \frac{1}{2}(x\overline{y}+y\overline{x}).
	\]
	In the definition of the inner product, the order of multiplication and conjugation can be exchanged quite freely. This follows from the previous proposition and the corresponding property of quaternions.
	\begin{cor}
		If $x,y\in\mathbb{O}$, then
		\[
		\SC(x)=\SC(\overline{x})
		\]
		and
		\[
		\SC(x\overline{y})= \SC(y\overline{x})=\SC(\overline{y}x)=\SC(y\overline{x}).
		\]
	\end{cor}
	Since $\SC(x\overline{x})=x\overline{x}$ for all $x\in\mathbb{O}$, the inner product induces a norm
	\[
	|x|^2:=x\overline{x}=\overline{x}x=\sum_{j=0}^7 x_j^2.
	\]
	As with the quaternion norm, the norm of octonions is a composition norm, i.e., formally
	\[
	|xy|^2=(xy)(\overline{xy})=(xy)(\overline{y}\, \overline{x})=(x\overline{x})(y\overline{y})=|x|^2|y|^2.
	\]
	The previous formula is formal because, as we see, we applied associativity just as with quaternions rather carelessly. The computation is correct, but requires the consideration of the following deeper algebraic results. We state the results without proof, but the reader can verify them by direct computation using the above techniques.
	
	\begin{prop}[$\mathbb{O}$ is an alternative algebra]\label{ALT}
		If $x,y\in \mathbb{O}$ then
		\[
		x(xy)=x^2y,\ (xy)y=xy^2,\ (xy)x=x(yx),
		\]
		
	\end{prop}
	
	\begin{prop}[Moufang Laws] For each $x,y,z\in\mathbb{O}$
		\[
		(xy)(zx)=(x(yz))x=x((yz)x).
		\]
	\end{prop}
	We observe that the composition property of the norm follows directly from the Moufang laws:
	
	\begin{prop}
		If $x,y\in\mathbb{O}$, then
		\[
		|xy|=|x||y|.
		\]
	\end{prop}
	From the definition of the norm it follows directly that every nonzero octonion is invertible:
	
	\begin{prop}
		The octonion algebra $\mathbb{O}$ is a division algebra. If $x\neq 0$, then
		\[
		x^{-1}:=\frac{\overline{x}}{|x|^2},
		\]
		which satisfies $xx^{-1}=x^{-1}x=1$.
	\end{prop}
	
	\subsection{Complexification of the Octonion Algebra}
	
	We consider the complexification of the octonion algebra $\mathbb{O}$:
	\[
	\mathbb{O}_{\mathbb{C}} := \mathbb{O} \otimes \mathbb{C}.
	\]
	The elements of the complexified octonion algebra are of the form
	\[
	z = x + i y,
	\]
	where $x, y \in \mathbb{O}$ and, in particular, $ix = xi$ for every $x \in \mathbb{O}$, that means $i$ commutates with each unit $e_j$ for all $j=0,1,\ldots,7$.  The multiplication in the algebra $\mathbb{O}_{\mathbb{C}}$ is naturally inherited from the product in the octonions. The complexified octonion algebra is neither a division algebra nor an alternative algebra. However, the complexified octonions satisfy alternativity with respect to octonionic coefficients, that is,
	\[
	x(xz) = x^2 z \quad \text{and} \quad (zx)x = zx^2
	\]
	for all $x \in \mathbb{O}$ and $z \in \mathbb{O}_{\mathbb{C}}$.
	We define the real and imaginary parts by setting
	\[
	\RE(z) = x \quad \text{and} \quad \IM(z) = y.
	\]
	Moreover, any element $z \in \mathbb{O}_{\mathbb{C}}$ can be decomposed as
	\[
	z = z_0 + \m{z},
	\]
	where $z_0 \in \mathbb{C}$ and $\m{z} \in \mathbb{C}^7$. The scalar and vector parts are defined by
	\[
	\SC(z) = z_0 \quad \text{and} \quad \VEC(z) = \m{z}.
	\]
	The complexified algebra contains more natural involutions than the original octonion algebra. We define the (algebraic) conjugate by setting
	\[
	\overline{z} = z_0 - \m{z}=\overline{x}+i\overline{y}.
	\]
	The complex conjugation is defined as
	\[
	z^* = x - i y.
	\]
	By combining the two previous involutions, we define the Hermitian conjugation as
	\[
	z^+ = \overline{z}^*.
	\]
	The conjugations satisfy the following multiplication rules:
	\[
	\overline{zw} = \overline{w}\;\overline{z}, \quad (zw)^* = z^* w^*, \quad \text{and} \quad (zw)^+ = w^+ z^+.
	\]
	The inner product on the octonion algebra can be extended to a complex inner product $\langle\cdot,\cdot\rangle_{\mathbb{C}}: \mathbb{O}_{\mathbb{C}} \times \mathbb{O}_{\mathbb{C}} \to \mathbb{C}$ by setting
	\[
	\langle z, w \rangle_{\mathbb{C}} := \SC(z w^+).
	\]
	A direct computation yields
	\[
	\langle z, w \rangle_{\mathbb{C}} = \langle x, u \rangle + \langle y, v \rangle + i (\langle y, u \rangle - \langle x, v \rangle),
	\]
	where $z = x + i y$ and $w = u + i v$. The complex inner product satisfies the standard axioms of an inner product:
	\begin{itemize}
		\item[(a)] $\langle z, z \rangle_{\mathbb{C}} > 0$ for all $z \neq 0$, 
		\item[(b)] $\langle z, z \rangle_{\mathbb{C}} = 0$ if and only if $z = 0$,
		\item[(c)] the inner product is conjugate symmetric, that is, $\langle z, w \rangle_{\mathbb{C}} = \langle w, z \rangle_{\mathbb{C}}^*$ for all $z, w \in \mathbb{O}_{\mathbb{C}}$,
		\item[(d)] the inner product is sesquilinear, i.e.,
		\[
		\langle \alpha_1 z_1 + \alpha_2 z_2, w \rangle_{\mathbb{C}} = \alpha_1 \langle z_1, w \rangle_{\mathbb{C}} + \alpha_2 \langle z_2, w \rangle_{\mathbb{C}}
		\]
		and
		\[
		\langle w, \alpha_1 z_1 + \alpha_2 z_2 \rangle_{\mathbb{C}} = \alpha_1^* \langle w, z_1 \rangle_{\mathbb{C}} + \alpha_2^* \langle w, z_2 \rangle_{\mathbb{C}}
		\]
		for all $z_1, z_2, w \in \mathbb{O}_{\mathbb{C}}$ and $\alpha_1, \alpha_2 \in \mathbb{C}$.
	\end{itemize}
	Moreover,
	\[
	\langle z^*, w^* \rangle_{\mathbb{C}} = \langle z, w \rangle_{\mathbb{C}}^* = \langle w, z \rangle_{\mathbb{C}}.
	\]
	We also observe that
	\[
	\langle iz, iw \rangle_{\mathbb{C}} = \langle z, w \rangle_{\mathbb{C}},
	\]
	which implies that the complex inner product defines a Kähler structure on the complexified octonion algebra $\mathbb{O}_{\mathbb{C}}$. As a consequence of sesquilinearity, we obtain the more general result
	\[
	\langle sz, sw \rangle_{\mathbb{C}} = \langle z, w \rangle_{\mathbb{C}}
	\]
	for all unit complex numbers $s \in S^1$.
	This invariance means that the inner product does not depend on the complex phase factor $s \in S^1$, so the complexified inner product is $U(1)$-invariant. This is a typical feature of Kähler geometry, where the complex structure, metric properties, and symplectic structure coexist in harmony.
	
	\begin{prop} Let $z=x+iy\in\mathbb{O}_{\mathbb{C}}$. Then 
		\[
		z\overline{z}=\overline{z}z=|x|^2-|y|^2+2i\langle x,y\rangle.
		\]    
	\end{prop}
	\begin{proof}
		We compute directly:
		\begin{align*}
			z\overline{z}&=(x+iy)(\overline{x}+i\overline{y}) \\
			&=x\overline{x}-y\overline{y}+i(y\overline{x}+x\overline{y})\\
			&=|x|^2-|y|^2+2i\langle x,y\rangle.
		\end{align*}
	\end{proof}
	Note that
	\begin{align*}
		|x|^2 - |y|^2 &= \RE(z\overline{z}),\\
		\langle x, y \rangle &= \frac{1}{2} \IM(z\overline{z}).
	\end{align*}
	Together, these formulas define a combined Euclidean–hyperbolic structure: the real part provides a Lorentz-type metric, while the imaginary part gives a Euclidean metric.  
	Together they form the foundation of a complex Hermitian metric. This reflects the fundamental idea of Kähler geometry:
	\begin{itemize}
		\item The metric incorporates both symplectic (imaginary part) and Riemannian (real part) components.
		\item The complex structure unifies these two components into a harmonious whole.
		\item $U(1)$-invariance emphasizes the central role of the complex structure.
	\end{itemize}
	In practice, this means that the complexified octonion algebra provides a geometric framework in which conformal geometry, symplectic geometry, and complex geometry naturally intersect.\\
	\\
	Let us say a few more words about the automorphism groups of the octonions and their complexifications.

	\begin{rem}
		If one compares the algebra-structure-preserving automorphisms of the real octonions and their complexification, the following facts emerge:
		\begin{itemize}
			\item $\mathrm{Aut}(\mathbb{O}_\mathbb{C}) \cong G_2(\mathbb{C})$, where $G_2(\mathbb{C})$ is the $14$-dimensional complex simple Lie group of type $G_2$,
			\item $\mathrm{Aut}(\mathbb{O}) \cong G_2(\mathbb{R})$, the compact real form of $G_2$.
		\end{itemize}
		These isomorphisms are consequences of the general theory; see \cite{SpringerVeldkamp2000}. In particular, Theorem 2.3.5 shows that the automorphism group of any octonion algebra over an arbitrary field is a simple algebraic group of type $G_2$, and Proposition 2.4.6 asserts that this group is defined over the base field. For the real case, the historical remarks in §2.5 explain that the compact real form of $G_2$ indeed appears as $\mathrm{Aut}(\mathbb{O})$. Interesting details on the aotmorphisms, isotropy groups and invariants of the octonions and complexified octonions have also been studied profoundly in the recent work by C. Bisi and J. Winkelmann \cite{BW} Section~9 and in Section~11, respectively.  
	\end{rem}

	\subsection{Octonion analysis}
	In octonion analysis, a central tool is the Cauchy–Riemann operator
	\[
	\partial_x=\partial_{x_0}+e_1\partial_{x_1}+\cdots+e_7\partial_{x_7}
	\]
	and its conjugate
	\[
	\partial_{\overline{x}}=\partial_{x_0}-e_1\partial_{x_1}-\cdots-e_7\partial_{x_7}
	\]
	The aforementioned operators act on differentiable octonion-valued functions $f:\Omega\to \mathbb{O}$, that is,
	\[
	f=f_0+f_1e_1+\cdots+f_7e_7,
	\]
	where $f_j:\Omega\to\mathbb{R}$ is differentiable, and $\Omega$ is an open subset of the space $\mathbb{R}^8\cong \mathbb{O}$. Octonions can also be considered bi-axially, as functions of two quaternionic variables $f=f(u,v)$, where $u,v\in\mathbb{H}$ such that $x=u+ve_4$. In that case, the Cauchy–Riemann operator can be written as
	\[
	\partial_ x=\partial_u+\partial_ve_4,
	\]
	where $\partial_u$ and $\partial_v$ are Cauchy–Riemann operators in the quaternionic algebra. Correspondingly, the conjugate is given by
	\[
	\partial_{\overline{x}}=\partial_{\overline{u}}-\partial_ve_4.
	\]
	This observation shows that the Cauchy–Riemann operator is a natural operator that doubles in the Cayley–Dickson process. For practical calculations, the following lemma is essential.
	
	\begin{lem}[\cite{KO}]\label{lem:diffrules}
		Let $f\colon\Omega\subset\mathbb{H}\to\mathbb{H}$ be a differentiable function and $\partial_u=\partial_{u_0}+e_1\partial_{u_1}+e_2\partial_{u_2}+e_3\partial_{u_3}$ the quaternionic Cauchy–Riemann operator. Then we have
		\begin{enumerate}[(a)]
			\item $\partial_u(fe_4)=(f\partial_u)e_4$,
			\item $(\partial_u e_4)f=(\partial_u\overline{f})e_4$,
			\item $(\partial_u e_4)(fe_4)=-\overline{f}\partial_u$.
			\item $(fe_4)\partial_u=(f\partial_{\overline{u}})e_4$,
			\item $f(\partial_u e_4)=(\partial_uf)e_4$,
			\item $(fe_4)(\partial_u e_4)=-\partial_{\overline{u}}f$.
		\end{enumerate}
	\end{lem}
	
	With the help of the above lemma, computations involving the Cauchy–Riemann operators of octonions can be carried out using quaternions. For example, if $f=g+he_4$ and as above $\partial_x=\partial_u+\partial_ve_4$, we obtain
	\begin{align}
		\partial_x f&=(\partial_u+\partial_ve_4)(g+he_4)\nonumber\\
		&=\partial_ug+\partial_u(he_4)+(\partial_ve_4)g+(\partial_ve_4)(he_4)\nonumber\\
		&=\partial_ug+(h\partial_u)e_4+(\partial_v\overline{g})e_4-\overline{h}\partial_v,\label{BICR}
	\end{align}
	
	Due to the alternativity of the octonion algebra, the Cauchy–Riemann operator factors the Laplace operator, i.e.,
	\begin{align}
		\partial_{x}\partial_{\overline{x}}=\partial_{\overline{x}}\partial_{x}=\Delta, \label{FUKTORISOINTI}  
	\end{align}
	where
	\[
	\Delta=\partial_{x_0}^2+\partial_{x_1}^2+\cdots+\partial_{x_7}^2
	\]
	is the Laplace operator in the Euclidean space $\mathbb{R}^8$. The above factorization formula can also be proven directly by computation using Lemma~\ref{lem:diffrules}. The following function class is central in octonion analysis.
	
	\begin{defn}
		Let $\Omega\subset\mathbb{R}^8$ be an open set and $f:\Omega\to\mathbb{O}$ a differentiable function. The function $f$ is (left) monogenic if
		\[
		\partial_xf=0
		\]
		in the domain $\Omega$. Similarly, $f$ is right monogenic if
		\[
		f\partial_x=0
		\]
		in the domain $\Omega$.
	\end{defn}
	
	A few immediate observations can be made about octonionic monogenic functions. First, from equation (\ref{BICR}) we see that monogenic functions satisfy a generalized quaternionic Cauchy–Riemann system:
	\[
	\partial_xf=0\qquad \Leftrightarrow\qquad \begin{cases}
		\partial_ug=\overline{h}\partial_v,\\
		h\partial_u=-\partial_v\overline{g}.
	\end{cases}
	\]
	when $f=g+he_4$.
	
	From the factorization formula (\ref{FUKTORISOINTI}) it follows that every monogenic function is also harmonic, i.e., if $\partial_xf=0$, then $\Delta f=0$.
	
	Moreover, if $f:\Omega\to\mathbb{O}$ is a harmonic function, i.e., $\Delta f=0$, then $\partial_{\overline{x}}f$ is monogenic by formula (\ref{FUKTORISOINTI}). This shows that the class of monogenic functions is rich, and one can speak of a nontrivial function theory.

	\section{Cauchy–Riemann Operators in Complexified Octonionic Analysis}
	In complexified octonionic analysis, one studies functions of the form
	\[
	f:\Omega \subset \mathbb{R}^{16} \to \mathbb{O}_{\mathbb{C}}.
	\]
	The variable in the function is considered to be a complexified octonion $z$, or equivalently, the function is often treated as a function of two octonionic variables:
	\[
	f = f(z) = f(x, y),
	\]
	where $z = x + i y$. Functions taking values in complex octonions naturally decompose as
	\[
	f = g + i h,
	\]
	where $g,h: \Omega \subset \mathbb{R}^{16} \to \mathbb{O}$ are octonion-valued functions of two octonionic variables. We define the complex Cauchy–Riemann operator by setting
	\[
	D = \frac{\partial}{\partial z_0} + e_1 \frac{\partial}{\partial z_1} + \cdots + e_7 \frac{\partial}{\partial z_7},
	\]
	where
	\[
	\frac{\partial}{\partial z_j} = \frac{\partial}{\partial x_j} + i \frac{\partial}{\partial y_j},
	\]
	for $j = 0, 1, \dots, 7$. It follows that
	\[
	D = \partial_x + i \partial_y,
	\]
	where $\partial_x$ and $\partial_y$ are the usual (real) octonionic Cauchy–Riemann operators.
	
	\begin{defn}
		A continuously differentiable function $f: \Omega \subset \mathbb{R}^{16} \to \mathbb{O}_{\mathbb{C}}$ is said to be left monogenic if $Df = 0$, and right monogenic if $fD = 0$.
	\end{defn}
	
	In what follows, we will consider only left monogenic functions, referred to simply as monogenic.
	
	\begin{prop}
		If $f = g + i h$ is differentiable, then the condition
		\[
		Df = 0
		\]
		is equivalent to
		\[
		\begin{cases}
			\partial_x g = \partial_y h, \\
			\partial_y g = -\partial_x h.
		\end{cases}
		\]
	\end{prop}
	\begin{proof}
		We compute
		\begin{align*}
			Df &= (\partial_x + i \partial_y)(g + i h) \\
			&= \partial_x g - \partial_y h + i(\partial_x h + \partial_y g).
		\end{align*}
	\end{proof}
	
	We define the Laplace operator corresponding to the real metric as
	\[
	\Delta = \Delta_x + \Delta_y,
	\]
	and the ultrahyperbolic Laplace operator as
	\[
	\widetilde{\Delta} = \Delta_x - \Delta_y.
	\]
	The Cauchy–Riemann operator factors the Euclidean Laplace operator up to a second-order operator of octonionic type. As a result, we obtain the following Lichnerowicz--Weitzenb\"ock type formula.
	
	\begin{prop} The Cauchy–Riemann operator satisfies 
		\[
		D D^+ = D^+ D = \Delta + i(\partial_y \partial_{\overline{x}} - \partial_x \partial_{\overline{y}}).
		\]
	\end{prop}
	\begin{proof}
		Using the alternativity of octonions, we compute:
		\begin{align*}
			D D^+ f &= (\partial_x + i \partial_y)(\partial_{\overline{x}} f - i \partial_{\overline{y}} f) \\
			&= \Delta_x f + \Delta_y f + i\big(\partial_y (\partial_{\overline{x}} f) - \partial_x (\partial_{\overline{y}} f)\big).
		\end{align*}
	\end{proof}
	
	It follows that the component functions of monogenic functions are no longer themselves monogenic. This marks a major distinction from (real) octonionic analysis. Moreover, the operator $\partial_y \partial_{\overline{x}} - \partial_x \partial_{\overline{y}}$ is octonion-valued, and thus the analysis of the Cauchy–Riemann operator $D$ becomes complicated due to this decomposition.\\
	\\
	By changing the involution in the decomposition, we obtain the following Lichnerowicz–Weitzenböck type formula, where the whole operator is scalar-valued.
	However, we observe that in this case one factors the ultrahyperbolic Laplace operator.
	
	\begin{prop} The Cauchy–Riemann operator satisfies 
		\[
		D \overline{D} = \overline{D} D = \widetilde{\Delta} + 2i \langle \partial_x, \partial_y \rangle.
		\]
	\end{prop}
	\begin{proof}
		Using octonionic alternativity, we compute:
		\begin{align*}
			D \overline{D} f &= (\partial_x + i \partial_y)(\partial_{\overline{x}} f + i \partial_{\overline{y}} f) \\
			&= \Delta_x f - \Delta_y f + i\big(\partial_y (\partial_{\overline{x}} f) + \partial_x (\partial_{\overline{y}} f)\big).
		\end{align*}
		We now simplify the latter term. Let $x = u + v e_4$ and $y = s + t e_4$. Without loss of generality, suppose $f = G + H e_4$ is an octonion-valued function where $G$ and $H$ are quaternion-valued. Applying Lemma~\ref{lem:diffrules}, we compute:
		\begin{align*}
			\partial_x f &= \partial_u G - \overline{H} \partial_v + \big(H \partial_u + \partial_v \overline{G}\big) e_4,
		\end{align*}
		and since $\overline{x} = \overline{u} - v e_4$, we get:
		\begin{align*}
			\partial_{\overline{x}} f &= \partial_{\overline{u}} G + \overline{H} \partial_v + \big(H \partial_{\overline{u}} - \partial_v \overline{G} \big) e_4.
		\end{align*}
		Now,
		\begin{align*}
			\partial_y (\partial_{\overline{x}} f)
			&= \partial_s (\partial_{\overline{u}} G + \overline{H} \partial_v)
			- \overline{(H \partial_{\overline{u}} - \partial_v \overline{G})} \partial_t \\
			&\quad + \big((H \partial_{\overline{u}} - \partial_v \overline{H}) \partial_s + \partial_t \overline{(\partial_{\overline{u}} G + \overline{H} \partial_v)} \big) e_4 \\
			&= \partial_s \partial_{\overline{u}} G + \partial_s \overline{H} \partial_v - \partial_u \overline{H} \partial_t + G \partial_{\overline{v}} \partial_t \\
			&\quad + \big(H \partial_{\overline{u}} \partial_s - \partial_v \overline{G} \partial_s + \partial_t \overline{G} \partial_u + \partial_t \partial_{\overline{v}} H \big) e_4,
		\end{align*}
		and similarly:
		\begin{align*}
			\partial_x (\partial_{\overline{y}} f)
			&= \partial_u \partial_{\overline{s}} G + \partial_u \overline{H} \partial_t - \partial_s \overline{H} \partial_v + G \partial_{\overline{t}} \partial_v \\
			&\quad + \big(H \partial_{\overline{s}} \partial_u - \partial_t \overline{G} \partial_u + \partial_v \overline{G} \partial_s + \partial_v \partial_{\overline{t}} H \big) e_4.
		\end{align*}
		Thus,
		\begin{align*}
			\partial_y (\partial_{\overline{x}} f) + \partial_x (\partial_{\overline{y}} f)
			&= 2 \langle \partial_s, \partial_u \rangle G + 2 \langle \partial_v, \partial_t \rangle G + 2\big( \langle \partial_u, \partial_s \rangle H + \langle \partial_t, \partial_v \rangle H \big) e_4 \\
			&=2\langle\partial_s,\partial_{u}\rangle f+2 \langle\partial_{v},\partial_t\rangle f
			= 2 \langle \partial_x, \partial_y \rangle f.
		\end{align*}
	\end{proof}
	
	If $f$ is a solution to the equation $D \overline{D} f = 0$, then $\overline{D} f$ is a monogenic function.  The following result provides a characterization of such functions:
	
	\begin{prop}
		Let $f = g + i h$ be twice differentiable, with $g$ and $h$ octonion-valued. Then $D \overline{D} f = 0$ if and only if
		\begin{align*}
			\widetilde{\Delta} g - 2 \langle \partial_x, \partial_y \rangle h = 0, \\
			\widetilde{\Delta} h + 2 \langle \partial_x, \partial_y \rangle g = 0.
		\end{align*}
	\end{prop}
	\begin{proof}
		We compute:
		\begin{align*}
			D \overline{D} f &= \widetilde{\Delta} f + 2i \langle \partial_x, \partial_y \rangle f 
			= \widetilde{\Delta} g + i \widetilde{\Delta} h + 2i \langle \partial_x, \partial_y \rangle g - 2 \langle \partial_x, \partial_y \rangle h.
		\end{align*}
	\end{proof}
	Let $\mathcal{L}:=D\overline{D}=\overline{D}D$. In the next section it will be observed that the operator $\mathcal{L}$ is ultrahyperbolic. 
	
	\begin{prop}
		If a function
		\[
		f=\sum_{j=0}^7 e_jf_j
		\]
		is monogenic, then its components belong to the kernel of the operator $\mathcal{L}$, that is, $\mathcal{L}f_j=0$ for each $j=0,1,...,7$.
	\end{prop}
	This leads to the following important remark:
	\begin{rem}
		The previous result determines one key difference between octonionic and complex-octonionic monogenic functions. The component functions of octonionic monogenic functions are harmonic, i.e., they belong to the kernel of an elliptic operator, whereas the components of complex-octonionic monogenic functions belong to the kernel of a ultrahyperbolic operator $\mathcal{L}$.
	\end{rem}

	\section{An integral representation of the fundamental solution}
	Since the second-order operator found in the previous section is of general interest, we momentarily move to the general case, that is, to arbitrary space dimensions.
	Let us consider the general differential operator
	\[
	\mathcal{L}:=\Delta_x-\Delta_y+2i\langle \partial_x,\partial_y\rangle=\sum_{j=1}^n \Big( \frac{\partial^2}{\partial x_j^2}-\frac{\partial^2}{\partial y_j^2}+2i \frac{\partial}{\partial x_j}\frac{\partial}{\partial y_j}\Big),
	\]
	where $n\in\mathbb{N}$. The complex octonionic setting is particular addressed for $n=8$. The operator thus acts on twice differentiable functions 
	\( f:\mathbb{R}^n \times \mathbb{R}^n \to \mathbb{R} \), 
	which depend on two vector variables \( f = f(x,y) \).\\
	\\
	We are looking for a distribution $E$, called the \textit{fundamental solution}, which satisfies
	\[
	\mathcal{L} E(x,y)=\delta(x,y).
	\]
	We first observe that if $\partial_{z_j}=\partial_{x_j}+i\partial_{y_j}$, then $\partial_{z_j}^2=\partial_{x_j}^2-\partial_{y_j}^2+2i\partial_{x_j}\partial_{y_j}$, so the operator can also be written in the form 
	\[
	\mathcal{L}=\sum_{j=1}^n \partial_{z_j}^2.
	\]
	
	\begin{prop} When $n=1$, a fundamental solution is of the form
		\[
		E(z)=\frac{1}{4\pi} \frac{z^*}{z}.
		\]
	\end{prop}
	\begin{proof}
		Now $\mathcal{L}=\partial_z^2=\partial_x^2-\partial_y^2+2i\partial_x\partial_y$. Since
		\[
		\partial_z(\frac{1}{2\pi z})=\delta(z),
		\]
		then
		\[
		\partial_zE=\frac{1}{2\pi z}
		\]
		the solution is
		\[
		E(z)=\frac{1}{4\pi} \frac{z^*}{z}
		\]
		for which $\mathcal{L}E=\delta(z)$.
	\end{proof}
	It is noted that the function found in the previous proposition is neither analytic nor anti-analytic.\\
	\\
	By applying the Fourier transform \(\mathcal{F}_{x\to\xi}\mathcal{F}_{y\to\eta}\) partially to the equation, we obtain
	\[
	(-|\xi|^2+|\eta|^2-2i\langle \xi,\eta\rangle)\widehat{E}(\xi,\eta)=1.
	\]
	The operator is ultrahyperbolic, since the symbol $P(\xi,\eta)=-|\xi|^2+|\eta|^2-2i\langle \xi,\eta\rangle=0$ when $|\xi|=|\eta|$ and $\langle \xi,\eta\rangle=0$. Formally applying the inverse transform, we obtain
	\[
	E(x, y) = \frac{1}{(2\pi)^{2n}} \int_{\mathbb{R}^{2n}} \frac{e^{i(\langle x, \xi \rangle + \langle y, \eta \rangle)}}{-|\xi|^2 + |\eta|^2 - 2i \langle \xi, \eta \rangle} \, d\xi \, d\eta.
	\]
	Since the Fourier integral contains an oscillatory exponent and a non-decaying denominator, the integral is interpreted as an oscillatory integral in the sense of distributions. This means that the integral does not converge as a usual Lebesgue integral. Typically, the integral is understood in the distributional sense as
	\[
	\left\langle E \mid \phi \right\rangle :=  \int_{\mathbb{R}^{2n}} \frac{1}{-|\xi|^2 + |\eta|^2 - 2i \langle \xi, \eta \rangle} \check{\phi}(\xi, \eta)\, d\xi \, d\eta,
	\]
	where \( \phi \in \mathcal{S}(\mathbb{R}^{2n}) \) is a Schwarz test function and $\check{\phi}$ its inverse Fourier transform. Let us write this integral using polar coordinates $\xi=r\omega$ and $\eta=\rho\nu$. Then we obtain
	\begin{align*}
		\left\langle E \mid \phi \right\rangle =  \int_{S^{n-1}} \int_{S^{n-1}} \int_0^\infty \int_0^\infty \frac{r^{n-1}\rho^{n-1}}{-r^2 + \rho^2 - 2i r \rho \langle \omega, \nu \rangle}\check{\phi}(r\omega, \rho\nu)\, drd\rho d\omega \, d\nu.
	\end{align*}
	This integral naturally splits into two parts according to the singularities of the denominator, namely
	\[
	\left\langle E \mid \phi \right\rangle=I_1(\phi)+I_2(\phi), 
	\]
	where
	\[
	I_1(\phi)=\iint_{\langle \omega, \nu \rangle\neq 0}  \int_0^\infty \int_0^\infty \frac{r^{n-1}\rho^{n-1}}{-r^2 + \rho^2 - 2i r \rho \langle \omega, \nu \rangle}\check{\phi}(r\omega, \rho\nu)\, drd\rho d\omega \, d\nu
	\]
	exists as an ordinary integral. The second part, in turn, is the principal value integral
	\[
	I_2(\phi)=\iint_{\langle \omega, \nu \rangle= 0}  P.V.\int_0^\infty \int_0^\infty \frac{r^{n-1}\rho^{n-1}}{-r^2 + \rho^2 }\check{\phi}(r\omega, \rho\nu)\, drd\rho d\omega \, d\nu,
	\]
	where we define
	\[
	P.V. \int_0^\infty \int_0^\infty  f(r,\rho) drd\rho:=\lim_{a\to\infty }\int_{\frac{1}{a}}^a \int_{\frac{1}{a} \le |\frac{\pi}{4}-\theta|\le \frac{\pi}{4}} f(R\cos(\theta),R\sin(\theta)) R d\theta dR.
	\]
	
	\begin{theorem}
		The integral
		\[
		I_2(\phi) = \iint_{\langle \omega, \nu \rangle = 0}  \mathrm{P.V.} \int_0^\infty \int_0^\infty \frac{r^{n-1}\rho^{n-1}}{-r^2 + \rho^2 }\check{\phi}(r\omega, \rho\nu)\, dr\, d\rho \, d\omega \, d\nu
		\]
		exists, where $\check{\phi} \in \mathcal{S}(\mathbb{R}^{2n})$ is a Schwartz function.
	\end{theorem}
	\begin{proof}
		Fix $\omega, \nu \in S^{n-1}$ such that $\langle \omega, \nu \rangle = 0$. Consider the integral
		\[
		J(\omega, \nu): = \mathrm{P.V.} \int_0^\infty \int_0^\infty \frac{r^{n-1}\rho^{n-1}}{-r^2 + \rho^2 }\check{\phi}(r\omega, \rho\nu)\, dr\, d\rho.
		\]
		Change to polar coordinates by setting
		\[
		r = R \cos \theta, \quad \rho = R \sin \theta, \quad R > 0, \quad \theta \in (0, \pi/2).
		\]
		Then
		\[
		dr\, d\rho = R\, dR\, d\theta, \quad r^{n-1}\rho^{n-1} = R^{2n-2} (\cos \theta \sin \theta)^{n-1}, \quad -r^2 + \rho^2 = R^2(\sin^2 \theta - \cos^2 \theta).
		\]
		The integral becomes
		\[
		J(\omega, \nu) = \lim_{a \to \infty} \int_{1/a}^{a} \int_{\substack{\theta \in (0, \pi/2) \\ |\pi/4 - \theta| \geq 1/a}} R^{2n-3} \frac{(\cos \theta \sin \theta)^{n-1}}{\sin^2 \theta - \cos^2 \theta} \check{\phi}(R \cos \theta \, \omega, R \sin \theta\, \nu)  \, d\theta \, dR.
		\]
		The singularity occurs when $\sin^2 \theta - \cos^2 \theta = 0$, i.e., $\theta = \pi/4$. Perform the change of variables $u = \theta - \pi/4$, so that
		\[
		\sin^2 \theta - \cos^2 \theta = \sin(2u), \quad \cos \theta \sin \theta = \frac{1}{2} \cos(2u).
		\]
		We obtain
		\[
		J(\omega, \nu) = \lim_{a \to \infty} \int_{1/a}^{a} \int_{|u| \geq 1/a} R^{2n-3} \left(\frac{1}{2}\right)^{n-1} \frac{(\cos(2u))^{n-1}}{\sin(2u)} \check{\phi}(R \cos(\pi/4+u)\, \omega, R \sin(\pi/4+u)\, \nu)  \, du \, dR.
		\]
		Set $t = 2u$, so that $du = dt/2$, and when $u \in [-\pi/4, \pi/4]$, then $t \in [-\pi/2, \pi/2]$. Then
		\[
		J(\omega, \nu) = \int_0^\infty R^{2n-3} \left(\frac{1}{2}\right)^{n} \left[ \mathrm{P.V.} \int_{-\pi/2}^{\pi/2} \frac{(\cos t)^{n-1}}{\sin t} \check{\phi}(R \cos(\pi/4 + t/2)\, \omega, R \sin(\pi/4 + t/2)\, \nu)  \, dt \right] \, dR.
		\]
		Define the function
		\[
		F(R, t) = \check{\phi}(R \cos(\pi/4 + t/2)\,  \omega, R \sin(\pi/4 + t/2)\, \nu).
		\]
		Since $\check{\phi}$ is a Schwartz function, $F$ is smooth and rapidly decreasing in $R$. In particular, for fixed $R$, $F(R, t)$ is smooth in $t$. Expand $F(R, t)$ in a Taylor series around $t = 0$:
		\[
		F(R, t) = F(R, 0) + F_t(R, 0) t + O(t^2).
		\]
		Similarly, $(\cos t)^{n-1} = 1 + O(t^2)$. Thus the integrand is
		\[
		\frac{(\cos t)^{n-1}}{\sin t} F(R, t) = \frac{1}{\sin t} \left( F(R, 0) + F_t(R, 0) t + O(t^2) \right).
		\]
		Since $\sin t \sim t$ near $t = 0$, we have
		\[
		\frac{(\cos t)^{n-1}}{\sin t} F(R, t) = \frac{F(R, 0)}{t} + F_t(R, 0) + O(t).
		\]
		The term
		\[
		\dfrac{F(R, 0)}{t}
		\]
		is an odd function, so its principal value over a symmetric interval is zero. Indeed, the principal value integral over $[-\pi/2, \pi/2]$ is defined by
		\[
		\mathrm{P.V.} \int_{-\pi/2}^{\pi/2} \frac{F(R, 0)}{t}  \, dt 
		= \lim_{\epsilon \to 0^+} \left( \int_{-\pi/2}^{-\epsilon} \frac{F(R, 0)}{t}  \, dt 
		+ \int_{\epsilon}^{\pi/2} \frac{F(R, 0)}{t}  \, dt \right).
		\]
		Computing the integrals:
		\[
		\int_{-\pi/2}^{-\epsilon} \frac{F(R, 0)}{t}  \, dt 
		= F(R, 0) \int_{-\pi/2}^{-\epsilon} \frac{1}{t}  \, dt 
		= F(R, 0) \left[ \ln |t| \right]_{-\pi/2}^{-\epsilon} 
		= F(R, 0) (\ln \epsilon - \ln \pi/2),
		\]
		and
		\[
		\int_{\epsilon}^{\pi/2} \frac{F(R, 0)}{t}  \, dt 
		= F(R, 0) \int_{\epsilon}^{\pi/2} \frac{1}{t}  \, dt 
		= F(R, 0) \left[ \ln |t| \right]_{\epsilon}^{\pi/2} 
		= F(R, 0) (\ln \pi/2 - \ln \epsilon).
		\]
		Combining these:
		\[
		\mathrm{P.V.} \int_{-\pi/2}^{\pi/2} \frac{F(R, 0)}{t}  \, dt 
		= \lim_{\epsilon \to 0^+} F(R, 0) \Big[ (\ln \epsilon - \ln \pi/2) + (\ln \pi/2 - \ln \epsilon) \Big] 
		= \lim_{\epsilon \to 0^+} F(R, 0) \cdot 0 = 0.
		\]
		The remaining terms are integrable, so
		\[
		\mathrm{P.V.} \int_{-\pi/2}^{\pi/2} \frac{(\cos t)^{n-1}}{\sin t} F(R, t)  \, dt
		\]
		exists for every fixed $R > 0$. Since $\check{\phi}$ is rapidly decreasing, this integral is rapidly decreasing in $R$, and thus the integral in $R$ converges. Hence $J(\omega, \nu)$ is well-defined. Now
		\[
		I_2(\phi) = \iint_{\langle \omega, \nu \rangle = 0} J(\omega, \nu)  \, d\omega \, d\nu.
		\]
		The set $\{ (\omega, \nu) \in S^{n-1} \times S^{n-1} : \langle \omega, \nu \rangle = 0 \}$ is compact. Since $\check{\phi}$ is smooth, $J(\omega, \nu)$ is continuous on this set. Therefore, the integral $I_2$ exists.
	\end{proof}
	
	\begin{rem}
		The notation $\iint_{\langle \omega, \nu \rangle = 0}$ does not mean an integral over a zero-measure set in the usual sense; rather, it is a shorthand emphasizing that the singularity occurs only when $\langle \omega, \nu \rangle = 0$. In practice, the integral is computed over the entire space $S^{n-1} \times S^{n-1}$, but the principal value handles the singularity at $\langle \omega, \nu \rangle = 0$. As a distribution, $I_2$ is defined as the limit
		\[
		I_2(\phi) = \lim_{\epsilon \to 0} \iint_{|\langle \omega, \nu \rangle| \geq \epsilon} 
		\left[ \mathrm{P.V.} \int_0^\infty \int_0^\infty 
		\frac{r^{n-1}\rho^{n-1}}{-r^2 + \rho^2}\,
		\check{\phi}(r\omega, \rho\nu)\, dr\, d\rho \right] d\omega\, d\nu,
		\]
		where the principal value inside the integral is already defined. The existence of this limit is ensured in the same way as in the proof.
	\end{rem}

	\section{On polynomial solutions}
	Consider the operator 
	\[
	\mathcal{L}=\sum_{j=1}^n \partial_{z_j}^2
	\]
	and its polynomial solutions. In this section, we use an overline to denote complex conjugation so that the notation does not become too cumbersome, in other words, if \( w = u + iv \), then \( \overline{w} = u - iv \).. There is no risk of confusion with octonion conjugation. Note that complex $k$-homogeneous polynomials are obtained from real polynomials
	\[
	\mathcal{P}_k(\mathbb{R}^{2n}) \otimes \mathbb{C} \cong \mathcal{P}_k(\mathbb{C}^n),
	\]
	where \((x_j, y_j) \mapsto (z_j, \overline{z}_j)\) is a linear bijection \(\mathbb{R}^{2n} \to \mathbb{C}^n\) with determinant \( (-2i)^n \neq 0\). Thus the monomials 
	\[
	\{ z^\alpha \overline{z}^\beta : |\alpha| + |\beta| = k \}
	\]
	form a basis for the space $\mathcal{P}_k(\mathbb{C}^n)$. Here $\alpha,\beta\in \mathbb{N}_0^n$ are multi-indices, i.e.\ for example $\alpha=(\alpha_1,...,\alpha_n)$, $z^\alpha:=z_1^{\alpha_1}\cdots z_n^{\alpha_n}$ and $|\alpha|=\alpha_1+\cdots+\alpha_n$. Through the previous isomorphism we obtain (see e.g.\ Proposition 5.8 in \cite{ABR})
	\[
	\dim \mathcal{P}_k(\mathbb{C}^n) = \binom{2n + k - 1}{2n-1}.
	\]
	\begin{prop}
		Let 
		\[
		P_k(z) = \sum_{\substack{\alpha, \beta \in \mathbb{N}_0^n \\ |\alpha| + |\beta| = k}} a_{\alpha,\beta}  z^\alpha \overline{z}^\beta
		\] 
		be a complex $k$-homogeneous polynomial. Then  
		\[
		\mathcal{L} P_k = 0 \quad \iff \quad \sum_{j=1}^n (\gamma_j + 2)(\gamma_j + 1)  a_{\alpha, \gamma + 2\varepsilon_j} = 0 \quad \forall \alpha, \gamma \text{ with } |\alpha| + |\gamma| = k-2,
		\]  
		where 
		\[
		\varepsilon_j=(0,...,0,\underbrace{1}_{j },0,...0)
		\]
		is the unit multi-index.
	\end{prop}
	
	\begin{proof}
		Since $\partial_{z_j}z_j=0$ and $\partial_{z_j}\overline{z}_j=2$, we obtain
		\[
		\partial_{z_j} (z^\alpha \overline{z}^\beta) = 2\beta_j  z^\alpha \overline{z}^{\beta - \varepsilon_j}, \quad \text{where } \beta_j \geq 1,
		\]  
		and correspondingly the second derivative
		\[
		\partial_{z_j}^2 (z^\alpha \overline{z}^\beta) = 4\beta_j (\beta_j - 1)  z^\alpha \overline{z}^{\beta - 2\varepsilon_j}, \quad \text{where } \beta_j \geq 2.
		\]  
		Using the operator $\mathcal{L} =  \sum_{j=1}^n \partial_{z_j}^2$ we get:  
		\[
		\mathcal{L} (z^\alpha \overline{z}^\beta) = 4 \sum_{j=1}^n \beta_j (\beta_j - 1)  z^\alpha \overline{z}^{\beta - 2\varepsilon_j}.
		\]  
		Thus
		\[
		\mathcal{L} P_k(z) = 4 \sum_{\substack{\alpha, \beta \\ |\alpha| + |\beta| = k}} a_{\alpha,\beta} \sum_{j=1}^n \beta_j (\beta_j - 1)  z^\alpha \overline{z}^{\beta - 2\varepsilon_j}.
		\]  
		Make the change of variable $\gamma = \beta - 2\varepsilon_j$, so that $\beta = \gamma + 2\varepsilon_j$ and $\beta_j = \gamma_j + 2$. Then  
		\[
		\mathcal{L} P_k(z) = 4 \sum_{\substack{\alpha, \gamma \\ |\alpha| + |\gamma| = k-2}} \left( \sum_{j=1}^n (\gamma_j + 2)(\gamma_j + 1)  a_{\alpha, \gamma + 2\varepsilon_j} \right) z^\alpha \overline{z}^\gamma.
		\]  
		The polynomial $\mathcal{L} P_k$ is identically zero if and only if all of its coefficients vanish. Hence  
		\[
		\mathcal{L} P_k(z) = 0 \quad \iff \quad \sum_{j=1}^n (\gamma_j + 2)(\gamma_j + 1)  a_{\alpha, \gamma + 2\varepsilon_j} = 0 \quad \forall \alpha, \gamma \text{ with } |\alpha| + |\gamma| = k-2.
		\]  
	\end{proof}
	
	\begin{thm}\label{POLYNMI}
		Let 
		\[
		\mathcal{H}_k^{\mathcal{L}}(\mathbb{C}^n) = \{ P_k \in \mathcal{P}_k(\mathbb{C}^n) : \mathcal{L} P_k = 0 \}
		\]
		be the space of $k$-homogeneous polynomials belonging to the kernel of the operator $\mathcal{L}$. Then
		\[
		\dim \mathcal{H}_k^{\mathcal{L}}(\mathbb{C}^n) = \binom{2n+k-1}{2n-1} - \binom{2n+k-3}{2n-1}, \quad k \ge 2,
		\]
		and \(\dim \mathcal{H}_0^{\mathcal{L}}(\mathbb{C}^n) = 1, \ \dim \mathcal{H}_1^{\mathcal{L}}(\mathbb{C}^n) = 2n\).
	\end{thm}
	\begin{proof}
		The space $\mathcal{P}_k(\mathbb{C}^n)$ consists of monomials $z^\alpha \overline{z}^\beta$ with $|\alpha| + |\beta| = k$, and its dimension is
		\[
		\dim \mathcal{P}_k(\mathbb{C}^n) = \binom{2n + k - 1}{2n-1}.
		\]
		Polynomials in the kernel of $\mathcal{L} = \sum_{j=1}^n \partial_{z_j}^2$ correspond to coefficients $a_{\alpha, \beta}$ satisfying the linear system
		\begin{align}\label{UHT}
			\sum_{j=1}^n (\gamma_j+2)(\gamma_j+1) a_{\alpha, \gamma + 2\varepsilon_j} = 0, \quad |\alpha| + |\gamma| = k-2.
		\end{align}
		Fix $\alpha$ and assume that $\gamma$ runs through all possible variations $\gamma^{(1)},...,\gamma^{(m)}$ such that $|\alpha|+|\gamma^{(j)}|=k-2$, for $j=1,...,m$. Assume further that the multi-indices are ordered lexicographically, i.e.\ $\gamma^{(1)}<\cdots<\gamma^{(m)}$. Then we obtain the equations
		\begin{align*}
			(\gamma_1^{(1)}+2)(\gamma_1^{(1)}+1) a_{\alpha, \gamma^{(1)} + 2\varepsilon_1}+\cdots+ (\gamma_n^{(1)}+2)(\gamma_n^{(1)}+1) a_{\alpha, \gamma^{(1)} + 2\varepsilon_n} = 0,\\
			(\gamma_1^{(2)}+2)(\gamma_1^{(2)}+1) a_{\alpha, \gamma^{(2)} + 2\varepsilon_1}+\cdots+ (\gamma_n^{(2)}+2)(\gamma_n^{(2)}+1) a_{\alpha, \gamma^{(2)} + 2\varepsilon_n} = 0,\\ 
			\cdots\\
			(\gamma_1^{(m)}+2)(\gamma_1^{(m)}+1) a_{\alpha, \gamma^{(m)} + 2\varepsilon_1}+\cdots+ (\gamma_n^{(m)}+2)(\gamma_n^{(m)}+1) a_{\alpha, \gamma^{(m)} + 2\varepsilon_n} = 0.\\
		\end{align*}
		The variables $a_{\alpha,\beta}$ satisfy $|\beta|=k-|\alpha|$, i.e.\ for each fixed $\alpha$ there are
		\[
		v=\dim  \mathcal{P}_{k-|\alpha|}(\mathbb{C}^n) = \binom{2n + k-|\alpha| - 1}{2n-1}
		\]
		variables. On the other hand, the equations for each fixed $\alpha$ correspond to those $\gamma$ with $|\gamma|=k-2-|\alpha|$. Thus
		\[
		m=\dim  \mathcal{P}_{k-2-|\alpha|}(\mathbb{C}^n) = \binom{2n + k-|\alpha| - 3}{2n-1}.
		\]
		Hence $m<v$, i.e.\ there are always fewer equations than variables.
	\end{proof}
	
	In the proof we need the following technical lemma.
	
	\begin{lem}
		Let
		\[
		B_j=\{ \beta=\gamma^{(j)}+2\varepsilon_\ell: \ell=1,...,n\}
		\]
		be the coefficients corresponding to equation $j$, for $j=1,...,m$. Then the index sets have at most one common element.
	\end{lem}
	\begin{proof}
		Assume that $\gamma^{(i)}<\gamma^{(j)}$ and suppose that
		\[
		\gamma^{(i)}+2\varepsilon_k=\gamma^{(j)}+2\varepsilon_\ell
		\]
		for some indices $k$ and $\ell$. If $k=\ell$, then $\gamma^{(i)}=\gamma^{(j)}$, a contradiction. It is possible to have one common element
		\[
		\gamma^{(i)}=\gamma^{(j)}+2\varepsilon_\ell-2\varepsilon_k
		\]
		when $k\neq \ell$. Assume that the sets
		\[
		B_i=\{\gamma^{(i)}+2\varepsilon_\ell\mid \ell=1,\dots,n\}
		\quad\text{and}\quad
		B_j=\{\gamma^{(j)}+2\varepsilon_k\mid k=1,\dots,n\}
		\]
		have two distinct common elements $\beta^{(1)}\neq\beta^{(2)}$. Then there exist indices
		$\ell_1,\ell_2,k_1,k_2$ such that
		\[
		\beta^{(1)}=\gamma^{(i)}+2\varepsilon_{\ell_1}=\gamma^{(j)}+2\varepsilon_{k_1},\qquad
		\beta^{(2)}=\gamma^{(i)}+2\varepsilon_{\ell_2}=\gamma^{(j)}+2\varepsilon_{k_2}.
		\]
		This implies
		\[
		\gamma^{(i)}-\gamma^{(j)}=2\varepsilon_{k_1}-2\varepsilon_{\ell_1},
		\qquad
		\gamma^{(i)}-\gamma^{(j)}=2\varepsilon_{k_2}-2\varepsilon_{\ell_2},
		\]
		and combining,
		\begin{align}\label{TAHTI}
			\varepsilon_{k_1}-\varepsilon_{\ell_1}=\varepsilon_{k_2}-\varepsilon_{\ell_2}. 
		\end{align}
		Equation (\ref{TAHTI}) is possible only in the following cases:
		\begin{enumerate} 
			\item $\ell_1=\ell_2$ and $k_1=k_2$. Then $\beta^{(1)}=\beta^{(2)}$, contradicting the assumption $\beta^{(1)}\neq\beta^{(2)}$.
			\item $\ell_1=k_2$ and $\ell_2=k_1$ (with $\ell_1\neq\ell_2$). Then equation (\ref{TAHTI}) simplifies to
			\[
			\varepsilon_{k_1}-\varepsilon_{\ell_1}=\varepsilon_{\ell_1}-\varepsilon_{k_1}
			\;\;\Longrightarrow\;\; 2(\varepsilon_{k_1}-\varepsilon_{\ell_1})=0,
			\]
			which is impossible since $k_1\neq \ell_1$ implies $\varepsilon_{k_1}\neq \varepsilon_{\ell_1}$.
			\item $\ell_1=k_1$ and $\ell_2=k_2$ (with $\ell_1\neq\ell_2$). Then from the original equations follows
			\[
			\gamma^{(i)}+2\varepsilon_{\ell_1}=\gamma^{(j)}+2\varepsilon_{\ell_1}
			\;\;\Longrightarrow\;\; \gamma^{(i)}=\gamma^{(j)},
			\]
			contradicting the assumption $\gamma^{(i)}\neq\gamma^{(j)}$.
			\item All four indices $\ell_1,\ell_2,k_1,k_2$ are distinct. Then the left and right side of (\ref{TAHTI}) affect different components and cannot be equal (e.g.\ in component $\ell_1$ the left side value is $-1$, right side $0$).
		\end{enumerate}
		In all cases we reach a contradiction. Thus the sets $B_i$ and $B_j$ can have at most one common element. This holds regardless of whether $\gamma^{(i)}< \gamma^{(j)}$ or not.
	\end{proof}
	\begin{proof}[Proof of Theorem \ref{POLYNMI} continued]
		Note that the coefficients $(\gamma_\ell^{(j)}+2)(\gamma_\ell^{(j)}+1)$ of the system of equations are all nonzero. For each index $j=1,...,m$, choose a multi-index
		\[
		\beta^{(j)}=\gamma^{(j)} + 2\varepsilon_{\ell_j},
		\]
		such that the corresponding variable $a_{\alpha,\beta^{(j)}}$ does not appear in the equations $1,...,j-1$. This is possible because, according to the previous lemma, the equations have at most one common variable, and there are fewer equations than variables.
		
		Denote the equations by $E_j$ and prove their linear independence inductively. The single equation $E_1$ is naturally linearly independent. Suppose that $E_1,...,E_{j-1}$ are linearly independent. Assume for contradiction that
		\[
		E_j=c_1E_1+\cdots + c_{j-1}E_{j-1}
		\]
		for some $c_1,...,c_{j-1}\in\mathbb{C}$. Then the variable $a_{\alpha,\beta^{(j)}}$, which appears only in $E_j$, does not appear on the right-hand side. This is a contradiction, and for each fixed $\alpha$ the equations are linearly independent. 
		
		Since for each fixed $\alpha$ the equations $E_1,...,E_m$ are linearly independent, all the equations (\ref{UHT}) are linearly independent. The number of equations is
		\[
		\sum_{|\alpha|\le k-2}\dim  \mathcal{P}_{k-2-|\alpha|}(\mathbb{C}^n) = \dim \mathcal{P}_{k-2}(\mathbb{C}^n)=\binom{2n + k - 3}{2n-1}.
		\]
		
		Hence, the number of linearly independent basis polynomials in the space $\mathcal{H}_k^{\mathcal{L}}(\mathbb{C}^n)$ is obtained by subtracting the number of linearly independent constraint equations from the total number of basis polynomials in the full space:
		\[
		\dim \mathcal{H}_k^{\mathcal{L}}(\mathbb{C}^n) = \dim \mathcal{P}_k(\mathbb{C}^n) - \dim \mathcal{P}_{k-2}(\mathbb{C}^n) = \binom{2n+k-1}{2n-1} - \binom{2n+k-3}{2n-1}.
		\]
		The cases $k=0,1$ are computed separately: 
		\begin{itemize}
			\item $k=0$: $\mathcal{H}_0^{\mathcal{L}}$ consists of constant polynomials, so $\dim \mathcal{P}_0(\mathbb{C}^n) = 1$.
			\item $k=1$: $\mathcal{L} P_1 \equiv 0$ for all $P_1 \in \mathcal{P}_1(\mathbb{C}^n)$ since the second derivatives vanish, hence $ \dim \mathcal{P}_1(\mathbb{C}^n)  = 2n$.
		\end{itemize}
	\end{proof}
	
	By the fundamental theorem of linear algebra, we have
	\[
	\dim \mathcal{P}_k(\mathbb{C}^n)=\dim \mathcal{H}_k^{\mathcal{L}}(\mathbb{C}^n) + \dim \text{Im}(\mathcal{L}),
	\]
	so by the previous proposition $\dim \text{Im}(\mathcal{L})=\dim \mathcal{P}_{k-2}(\mathbb{C}^n)$, which implies:
	\begin{cor}
		The operator $\mathcal{L}: \mathcal{P}_k(\mathbb{C}^n)\to  \mathcal{P}_{k-2}(\mathbb{C}^n)$ is surjective for $k\ge 2$. 
	\end{cor}
	By direct calculation, we have the examples
	\begin{align*}
		\mathcal{H}_2^{\mathcal{L}}(\mathbb{C})&=\text{span}\{ z^2,z\overline{z}\},\\ 
		\mathcal{H}_3^{\mathcal{L}}(\mathbb{C})&=\text{span}\{z^3, z^2\overline{z}\},\\
		\mathcal{H}_4^{\mathcal{L}}(\mathbb{C}) &= \text{span}\{ z^4,\, z^3 \overline{z} \},\\
		\mathcal{H}_2^{\mathcal{L}}(\mathbb{C}^2)
		&=\text{span}\{ z_1^2 , z_2^2, z_1\overline{z}_1, z_1\overline{z}_2, z_2\overline{z}_1, z_2\overline{z}_2,  \overline{z}_1\overline{z}_2, z_1z_2, \overline{z}_1^2-\overline{z}_2^2\},\\ 
		\mathcal{H}_3^{\mathcal{L}}(\mathbb{C}^2)&=\text{span}\Big\{
		z_1^3,\, z_2^3,\, z_1^2 z_2,\, z_1 z_2^2,\, 
		z_1^2 \overline{z}_1,\, z_1 z_2 \overline{z}_1,\, z_2^2 \overline{z}_1,\, 
		z_1^2 \overline{z}_2,\, z_1 z_2 \overline{z}_2,\, z_2^2 \overline{z}_2,\, 
		z_1 \overline{z}_1 \overline{z}_2,\, z_2 \overline{z}_1 \overline{z}_2,\\
		&\qquad  \qquad  \overline{z}_1^3 - 3 \overline{z}_1 \overline{z}_2^2,\,
		3\overline{z}_1^2 \overline{z}_2 -  \overline{z}_2^3,\,
		z_1 \overline{z}_1^2 - z_1 \overline{z}_2^2,\,
		z_2 \overline{z}_1^2 - z_2 \overline{z}_2^2
		\Big\},\\
		\mathcal{H}_4^{\mathcal{L}}(\mathbb{C}^2) &=
		\text{span} \Big\{
		z_1^4,\, z_1^3 z_2,\, z_1^2 z_2^2,\, z_1 z_2^3,\, z_2^4,\,
		z_1^3 \overline{z}_1,\, z_1^2 z_2 \overline{z}_1,\, z_1 z_2^2 \overline{z}_1,\, z_2^3 \overline{z}_1,\,
		z_1^3 \overline{z}_2,\, z_1^2 z_2 \overline{z}_2,\, z_1 z_2^2 \overline{z}_2,\, z_2^3 \overline{z}_2,\,
		z_1^2 \overline{z}_1 \overline{z}_2,\, z_1 z_2 \overline{z}_1 \overline{z}_2,\, z_2^2 \overline{z}_1 \overline{z}_2\\
		&\qquad  \qquad \overline{z}_1^4 - 6 \overline{z}_1^2 \overline{z}_2^2 + \overline{z}_2^4,\,
		\overline{z}_1^3 \overline{z}_2 - \overline{z}_1 \overline{z}_2^3,\,
		z_1 \overline{z}_1^3 - 3 z_1 \overline{z}_1 \overline{z}_2^2, \,
		3z_1 \overline{z}_1^2 \overline{z}_2 -  z_1 \overline{z}_2^3, \,
		z_2 \overline{z}_1^3 - 3 z_2 \overline{z}_1 \overline{z}_2^2, \,
		3z_2 \overline{z}_1^2 \overline{z}_2 -  z_2 \overline{z}_2^3, \\
		&\qquad  \qquad z_1^{2}\overline{z}_1^{2}-z_1^{2}\overline{z}_2^{2},\,  z_2^{2}\overline{z}_1^{2}-z_2^{2}\overline{z}_2^{2},\, z_1 z_2 \overline{z}_1^{2}-z_1 z_2 \overline{z}_2^{2}
		\Big\}.
	\end{align*}
	Although the operator $\mathcal{L}: \mathcal{P}_k(\mathbb{C}^n)\to  \mathcal{P}_{k-2}(\mathbb{C}^n)$ is surjective, the Fischer decomposition
	\[
	\mathcal{P}_k(\mathbb{C}^n)= Q(z)\mathcal{H}_k^{\mathcal{L}}(\mathbb{C}^n) \oplus \mathcal{P}_{k-2}(\mathbb{C}^n)
	\]
	where $Q(z) = \sum_{j=1}^n z_j^2$, does not hold. More precisely, the operator $\mathcal{L}$ does not determine the Fischer decomposition because it is not elliptic. The following counterexample can be constructed:
	
	\begin{example} Let $n=1$ and $k=2$. In this situation, the dimension of the 2-homogeneous polynomials is 3, and using the previous example, we obtain
		
		\begin{align*}
			\mathcal{P}_2(\mathbb{C}) &= \operatorname{span}\{ z^2, z\overline{z}, \overline{z}^2 \},\\
			\mathcal{H}_2^{\mathcal{L}}(\mathbb{C}) &= \operatorname{span}\{ z^2, z\overline{z} \},\\
			Q(z) \mathcal{P}_0(\mathbb{C}) &= z^2 \cdot \mathbb{C} = \operatorname{span}\{ z^2 \},\\
			\mathcal{H}_2^{\mathcal{L}}(\mathbb{C}) \cap Q(z) \mathcal{P}_0(\mathbb{C}) &= \operatorname{span}\{ z^2 \} \neq \{0\}.
		\end{align*}
	\end{example}
	For the octonionic case, we have
	\begin{align*}
		\dim \mathcal{H}_1^{\mathcal{L}}(\mathbb{C}^8)&=16,\\
		\dim \mathcal{H}_2^{\mathcal{L}}(\mathbb{C}^8)&=135,\\
		\dim \mathcal{H}_3^{\mathcal{L}}(\mathbb{C}^8)&=800,\\
		\dim \mathcal{H}_4^{\mathcal{L}}(\mathbb{C}^8)&=3740.\\
	\end{align*}
	Define octonionic polynomial solutions of the operator 
	\[
	\mathcal{L}=D\overline{D}=\overline{D}D=\sum_{j=0}^7 \partial_{z_j}^2
	\]
	by setting
	\[
	\mathcal{H}_k^{\mathcal{L}}(\mathbb{O}):=\mathcal{H}_k^{\mathcal{L}}(\mathbb{C}^8)\otimes \mathbb{O},
	\]
	i.e., $P_k\in \mathcal{H}_k^{\mathcal{L}}(\mathbb{O})$ if and only if
	\[
	P_k(z)=\sum_{j=0}^7 a_j P_k^{(j)}(z),
	\]
	where $P_k^{(j)}\in \mathcal{H}_k^{\mathcal{L}}(\mathbb{C}^8)$ and $a_j\in\mathbb{O}$ for $j=0,1,...,7$. Then $k$-homogeneous monogenic polynomial solutions for the operator $D$ can be constructed by setting
	\[
	Q_k(z):=\overline{D}P_{k+1}(z),
	\]
	where $P_{k+1}\in \mathcal{H}_{k+1}^{\mathcal{L}}(\mathbb{O})$. We leave the detailed study of these polynomials for later. 
	
	\begin{rem}
		Note that the operator $\mathcal{L}$ has the same number of $k$-homogeneous polynomial solutions in the space $\mathbb{R}^{2n}$ as the corresponding $k$-homogeneous harmonic polynomials. See Proposition 5.8 in \cite{ABR} for details.
	\end{rem}
	Finally, we would like to point out to the reader what lies behind the consideration of all polynomials. In this way, we aim to convince that the function theory of complexified octonions contains a wealth of interesting functions. This is always a fundamental point that it is good to state when developing any new function theory.
	
	\subsection*{Acknowledgement}
	This research article originated while the second author was visiting the University of Erfurt in the summer of 2025. He wishes to thank the university staff for the inspiring atmosphere and their highly professional attitude toward mathematics.
	
	\section*{Declarations}
	
	\subsection*{Availability of data and material}
	All data generated or analysed during this study are included in this published article.
	
	\subsection*{Competing interests}
	The authors declare that they have no competing interests.
	
	\subsection*{Funding}
	This research received no external funding.
	
	\subsection*{Authors' contributions}
	Both authors contributed equally to the conception, analysis, and writing of this manuscript.


\begin{thebibliography}{99}
		
		\bibitem{ABR} S. Axler, P. Bourdon, and W. Ramey, \textit{Harmonic Function Theory}, 2nd ed., Springer, New York, 2001.
		
		\bibitem{Baez}
		J. Baez, \emph{The octonions}, Bull. Amer. Math. Soc. \textbf{39}(2) (2002), 145--205.
		
		\bibitem{BW} C. Bisi, and J. Winkelmann, Automorphisms for slice-regular functions: the octonionic case (2024). https://arxiv.org/pdf/2411.16762. To appear
		
		\bibitem{Bosshard1939} P. Bo{\ss}hard. Die Cliffordschen Zahlen, ihre Algebra und ihre Funktionentheorie, PhD Thesis, Univ. Z\"urich, 1940.
		
		
		\bibitem{BDS}
		F. Brackx, R. Delanghe, and F. Sommen \emph{Clifford Analysis}, Pitman Res. Notes in Math.-Vol.76, Boston, 1982.
		
		\bibitem{BSS} 
		
		F. Brackx, H. De Schepper and F. Sommen. The Hermitian Clifford analysis toolbox. {\it Adv. Appl. Clifford Algebr.} {\bf 18}(3–4) (2008), 451–-487. 
		
		
		\bibitem{Burdik}
		C. Burdik, S. Catto, Y. G\"urcan, A. Khalfan, L. Kurt, and V. Kato La. \emph{$SO(9,1)$ group and examples of analytic functions} Journal of Physics: Conference Series-Vol.1194 (2019), Article No. 012016.
		
		
		\bibitem{CKS2023} F. Colombo, R.S. Krau{\ss}har, and I. Sabadini, \emph{Octonionic monogenic and slice monogenic Hardy and Bergman spaces},  Forum Mathematicum, in press, https://doi.org/10.1515/forum-2023-0039.
		
		
		
		\bibitem{ConKra2021}
		D. Constales, R.S. Krau{\ss}har, \emph{Octonionic Kerzman-Stein operators}, Compl. Anal. Oper. Theory \textbf{15}(6) (2021), Article No. 104 (23pp.).
		
		
		
		\bibitem{DS}
		P. Dentoni and M. Sce, \emph{Funzioni regolari nell’algebra di Cayley}, Rend. Sem. Mat. Univ. Padova \textbf{50} (1973), 251--267.
		
		\bibitem{Dinh} D. C. Dinh. Cauchy–Riemann operator in Cayley–Dickson–Clifford analysis, Boletin de la Sociedad Mexicana {\bf 30} (2024) Article Number 89.
		
		\bibitem{FL}
		E. Frenod and S. V. Ludkovski, \emph{Integral operator approach over octonions to solution of nonlinear PDE}, Far East J. Math. Sci. \textbf{103}(5) (2018), 831--876.
		
		\bibitem{Fueter1949} R. Fueter. Functions of a Hyper Complex Variable, Lecture notes written and supplemented by E. Bareiss, Math. Institut, Univ. Z\"urich, 1948/49
		
		\bibitem{GH}
		H. H. Goldstine and L. P. Horwitz, \emph{Hilbert space with non-associative scalars I}, Math. Ann. \textbf{154}(1) (1964), 1--27.
		
		\bibitem{GS2}
		K. G\"urlebeck and W. Spr\"o{\ss}ig, \emph{Quaternionic and Clifford calculus for physicists and engineers}, Mathematical Methods in Practice, Wiley, Chichester, 1997.
		
		
		\bibitem{QR2021}
		Q. Huo and G. Ren, \emph{Para-linearity as the nonassociative counterpart of linearity}, J. Geom. Anal. \textbf{32}(12) (2022), Article No. 304 (30pp.).
		
		\bibitem{QR2022}
		Q. Huo and G. Ren, \emph{Structure of octonionic Hilbert spaces with applications in the Parseval Equality and Cayley-Dickson algebras}, J. Math. Phys. {63}(4) (2022), Article No. 042101 (24pp.).
		
		\bibitem{KO} J. Kauhanen and H. Orelma, \textit{Cauchy-Riemann Operators in Octonionic Analysis},  Adv. Appl. Clifford Algebras 28, 1 (2018)
		
		\bibitem{Kauhanen_3}
		J. Kauhanen and H. Orelma, \emph{On the structure of octonion regular functions}, Adv. Appl. Clifford Algebr. \textbf{29}(4) (2019), Article No. 77 (17pp.).
		
		\bibitem{KraMMAS}
		R.S. Krau{\ss}har, \emph{Recent and new results on octonionic Bergman and Szeg\"o kernels},  Math. Meth. Appl. Sci. (2021), 1--14, https://doi.org/10.1002/mma.7316, 14pp.
		
		
		
		\bibitem{KFVR} R.S. Krau{\ss}har, M. Ferreira, N. Vieira, M.M. Rodrigues, The Teodorescu and the $\Pi$-operator in octonionic analysis and some applications, Journal of Geometry and Physics, {\bf 206} (2024), 105328
		
		
		\bibitem{XLT2008}
		X. Li, L. Peng, and T. Qian, \emph{Cauchy integrals on Lipschitz surfaces in octonionic space}, J. Math. Anal. Appl. \textbf{343}(2) (2008), 763--777.
		
		\bibitem{XL2002}
		X. Li and L. Peng, \emph{The Cauchy integral formulas on the octonions}, Bull. Belg. Math. Soc. \textbf{9}(1)  (2002), 47--62.
		
		
		
		\bibitem{XL2000}
		X. Li and L. Peng,  \emph{On Stein-Weiss conjugate harmonic function and octonion analytic function},  Approx. Theory Appl. \textbf{16} (2000), 28--36.
		
		\bibitem{LRW} Yong Li, Guangbin Ren and Haiyan Wang. Expliciz Witt basis over the tesor product of Clifford algebras and octonions, Preprint 2025, https://arxiv.org/pdf/2404.03487v1
		
		\bibitem{Nono}
		K. N\^{o}no, \emph{On the octonionic linearization of Laplacian and octonionic function theory}, Bull. Fukuoka Univ. Ed. Part III \textbf{37} (1988), 1--15.
		
		\bibitem{Prather2021}
		B. Prather, \emph{Octonions -- Hilbert spaces, fibrations and analysis}, PhD Thesis, Florida State University, 2021.
		
		\bibitem{Ryan} J. Ryan. Complexified Clifford Analysis, Complex Analysis 1(1982), 119--149.
		
		\bibitem{SaSo2002} I. Sabadini and F. Sommen.  Hermitian Clifford analysis and resolutions, {\it Math. Meth. Appl. Sci.} 
		{\bf 25}(2002), 1395--1413.
		
		\bibitem{SpringerVeldkamp2000} T. A. Springer and F. D. Veldkamp, \emph{Octonions, Jordan Algebras and Exceptional Groups}, Springer Monographs in Mathematics, Springer-Verlag, Berlin, Heidelberg, 2000. 
		
		\bibitem{WL2018}
		J. Wang and X. Li, \emph{The octonionic Bergman kernel for the unit ball}, Adv. Appl. Clifford Algebras \textbf{28}(3) (2018), Article No. 60 (10pp.).
		
		\bibitem{WL2020}
		J. Wang and X. Li, \emph{The octonionic Bergman kernel for the half space}, Adv. Appl. Clifford Algebras \textbf{30}(4) (2020), Article No. 57 (11pp.)
		
		
		
		
	\end{thebibliography}
\end{document}